%% file: Host-alternate-proof-revised-2.tex
\documentclass[english]{article}
\usepackage[T1]{fontenc}
\usepackage{amsmath}
\usepackage{amsthm}
\usepackage{amssymb}

\makeatletter
\numberwithin{equation}{section}
\theoremstyle{plain}
\newtheorem{thm}{\protect\theoremname}[section]
\theoremstyle{plain}
\newtheorem{lem}[thm]{\protect\lemmaname}
\theoremstyle{plain}
\newtheorem{cor}[thm]{\protect\corollaryname}
\theoremstyle{plain}
\newtheorem{prop}[thm]{\protect\propositionname}
\theoremstyle{remark}
\newtheorem{rem}[thm]{\protect\remarkname}

\input{./tex-shortcuts.tex}

\date{}

\makeatother

\usepackage{babel}
\providecommand{\corollaryname}{Corollary}
\providecommand{\lemmaname}{Lemma}
\providecommand{\propositionname}{Proposition}
\providecommand{\remarkname}{Remark}
\providecommand{\theoremname}{Theorem}

\begin{document}
\title{A short proof of Host's equidistribution theorem}
\author{Michael Hochman\thanks{Supported by ISF grants 1702/17 and 3056/21}}
\maketitle
\begin{abstract}
This note contains a new proof of Host's equidistribution theorem
for multiplicatively independent endomorphisms of $\mathbb{R}/\mathbb{Z}$.
The method is a simplified version of our recent work on equidistribution
under toral automorphisms \cite{Hochman2019} and is related to the
argument in \cite{HochmanShmerkin2015-equidistribution-from-fractal-measures},
but avoids the use of the scenery flow and of Marstrand's projection
theorem, using instead a direct Fourier argument to establish smoothness
of the limit measure.
\end{abstract}
\tableofcontents{}

\section{Introduction}

Furstenberg has famously conjectured that if $a,b\in\mathbb{N}$ are
multiplicatively independent integers, then the only Borel probability
measures on $\mathbb{R}/\mathbb{Z}$ that are invariant and ergodic
under $\times a$ and $\times b$ are either atomic or Lebesgue. The
conjecture has been partially verified by Rudolph and Johnson under
an assumption of positive entropy; many generalizations exist.

Closely related to Rudolph's theorem is Host's equidistribution theorem:\footnote{Host \cite{Host95} proved this assuming $\gcd(a,b)=1$. A more general
statement, when $a\not|b^{k}$ for all $k$, was proved by Lindenstrauss
\cite{Lindenstrauss2001b}, and the general result by Hochman-Shmerkin
\cite{HochmanShmerkin2015-equidistribution-from-fractal-measures}.}
\begin{thm}
[\cite{Host95,Lindenstrauss2001b,HochmanShmerkin2015-equidistribution-from-fractal-measures}]
If $\mu$ is a probability measure on $\mathbb{R}/\mathbb{Z}$ that
is invariant, ergodic and has positive entropy under an endomorphism
$\times a$, then $\mu$-a.e. point equidistributes for Lebesgue measure
under $\times b$, provided $a$ and $b$ are multiplicatively independent. 
\end{thm}

In this note we give a new proof of Host's theorem. The proof is a
simplified version of the one in \cite{Hochman2019}, but is also
related to the proof from \cite{HochmanShmerkin2015-equidistribution-from-fractal-measures}.
It differs from the latter primarily in its ``endgame'', where the
use of Marstrand's projection theorem is replaced by a more direct
Fourier-theoretic argument, and it avoids the use of the scenery flow
machinery found there. We also note that unlike Host's original proof,
the one here covers the general result for multiplicatively independent
endomorphisms.

\section{\label{sec:General-results-on-equidistribution}General results on
equidistribution}

In this section we establish some ``soft'' equidistribution results. 

If $(X,\mathcal{B})$ is a standard Borel space and $\mathcal{A}\subseteq\mathcal{B}$
is a measurable partition (or countably generated $\sigma$-algebra),
\,we write $\mathcal{A}(x)$ for the unique element of $\mathcal{A}$
containing a point $x\in X$. If $\mu$ is a probability measure and
$\mu(A)>0$ then $\mu_{A}=\frac{1}{\mu(A)}\mu|_{A}$ denotes the normalized
restriction of $\mu$ to $A$. If $\mathcal{C\subseteq\mathcal{B}}$
is a countably generated $\sigma$-algebra, then $\mu_{x}^{\mathcal{C}}$
denotes the conditional measure of $\mu$ on $\mathcal{C}(x)$, which
is defined $\mu$-a.e.

\subsection{\label{subsec:The-ergodic-theorem-for-M-differences}The ergodic
theorem for martingale differences}

We record a variant of the ergodic theorem for martingale differences,
and a simple consequence.
\begin{thm}
\label{thm:martingale-differences}Let $\mathcal{B}_{1}\subseteq\mathcal{B}_{2}\subseteq\mathcal{B}_{3}\subseteq\ldots\subseteq\mathcal{B}$
be an increasing sequence of $\sigma$-algebras in a probability space
$(X,\mathcal{B},\mu)$. Let $f_{n}\in L_{\infty}(\mu,\mathcal{B}_{n+1})$
be a uniformly bounded sequence. Then with probability one,
\begin{equation}
\lim_{N\rightarrow\infty}\frac{1}{N}\sum_{n=1}^{N}(f_{n}-\mathbb{E}(f_{n}\,|\,\mathcal{B}_{n}))=0\label{eq:6}
\end{equation}
More generally, suppose $k\in\mathbb{N}$ and $f_{n}\in L_{\infty}(\mu,\mathcal{B}_{n+k})$.
Then (\ref{eq:6}) still holds.
\end{thm}

\begin{proof}
The first part (when $k=1$) is the standard version \cite[Chapter 7, Theorem 3]{Feller71}.
Now suppose $k\neq1$. For each $p=0,1,2,\ldots,k-1$, apply the standard
version of the theorem to the increasing sequence of $\sigma$-algebras
$(\mathcal{B}_{kn+p})_{n=1}^{\infty}$, obtaining the a.s. limit
\[
\lim_{N\rightarrow\infty}\frac{1}{N}\sum_{n=1}^{N}(f_{kn+p}-\mathbb{E}(f_{kn+p}\,|\,\mathcal{B}_{kn+p}))=0
\]
Averaging these limits over $p$ gives (\ref{eq:6}).
\end{proof}

\subsection{\label{subsec:Relating-orbits-to-structure-of-mu}Relating orbits
to the local structure of $\mu$ }

Let $T$ be a measurable transformation of a compact metric space
and $\mu$ a Borel probability measure on $X$. We do not assume that
$\mu$ is preserved by $T$. Our goal is to describe the statistical
behavior of the orbit of a $\mu$-typical point $x$. In \cite[Theorem 2.1]{HochmanShmerkin2015-equidistribution-from-fractal-measures},
we showed that if $\mathcal{A}$ is a generating partition for $T$
and $\mathcal{A}^{n}(x)=(\bigvee_{i=0}^{n-1}T^{-i}\mathcal{A})(x)$,
then any measure $\nu$ for which the orbit equidistributes (possibly
along a subsequence) can be described as a limit of averages of the
measures $T^{n}(\mu_{\mathcal{P}^{n}(x)})$.

In the present work we use a different version in which $\mathcal{A}$
is adapted to the dynamics of a different map, or from some hierarchical
structure of $\mu$. We require the atoms of $\mathcal{A}^{n}$ must
have some compatibility with the expansion of $T^{n}$. 
\begin{thm}
\label{thm:equidistribution-1}Let $T:X\rightarrow X$ be a continuous
map of a compact metric space. Let $\mathcal{A}_{1},\mathcal{A}_{2},\mathcal{A}_{3},\ldots$
be a refining sequence of Borel partitions. Let $\mu$ be a Borel
probability measure on $X$ and assume that 
\begin{equation}
\sup_{n\in\mathbb{N}}\{\diam T^{n}A\,:\,A\in\mathcal{A}_{n+k}\,,\,\mu(A)>0\}\rightarrow0\qquad\text{ as }k\rightarrow\infty\label{eq:5}
\end{equation}
Then for $\mu$-a.e. $x$,
\[
\lim_{N\rightarrow\infty}\left(\frac{1}{N}\sum_{n=1}^{N}\delta_{T^{n}x}-\frac{1}{N}\sum_{n=1}^{N}T^{n}\mu_{\mathcal{A}_{n}(x)}\right)=0
\]
in the weak-{*} sense.
\end{thm}

\begin{proof}
Let $\mathcal{F}\subseteq C(X)$ be a dense countable set. It is enough
to prove that for every $f\in\mathcal{F}$ , for $\mu$-a.e. $x$,
\begin{equation}
\lim_{N\rightarrow\infty}\left(\frac{1}{N}\sum_{n=1}^{N}f(T^{n}x)-\frac{1}{N}\sum_{n=1}^{N}\int fdT^{n}\mu_{\mathcal{A}_{n}(x)}\right)=0\label{eq:7}
\end{equation}
Fix $f\in\mathcal{F}$. Our assumption (\ref{eq:5}) implies the $f\circ T^{n}-\mathbb{E}(f\circ T^{n}\,|\,\mathcal{A}_{n+k})\rightarrow0$
as $k\rightarrow\infty$, uniformly in $n\in\mathbb{N}$ and $x\in\supp\mu$.
Therefore it suffices, for each $k$, to prove (\ref{eq:7}) with
$\mathbb{E}(f\circ T^{n}\,|\,\mathcal{A}_{n+k})$ in place of $f\circ T^{n}$.
We can re-write the other term (\ref{eq:7}) as
\begin{align*}
\int fdT^{n}\mu_{\mathcal{A}_{n}(x)} & =\int f\circ T^{n}\:d\mu_{\mathcal{A}_{n}(x)}\\
 & =\mathbb{E}_{\mu}(f\circ T^{n}\,|\,\mathcal{A}_{n})(x)
\end{align*}
With these modifications, (\ref{eq:7}) follows directly from the
results of the previous section applied to the functions $f_{n}=\mathbb{E}(f\circ T^{n}\,|\,\mathcal{A}_{n+k})$.
\end{proof}

\subsection{\label{subsec:Joinings-with-group-rotations}Equidistribution along
the times $[\beta n]$}

We will need an equidistribution result for pairs of orbits of the
form $(\theta n,T^{[\beta n]}x)$ where $\theta\in\mathbb{R}/\mathbb{Z}$
and $x$ is a typical point for the measure preserving map $T$. The
argument is rather standard but we record the proof for completeness.
We generically write $R_{\theta}$ to denote translation by $\theta$. 
\begin{lem}
\label{lem:pairs}Let $X,Y$ be compact metric spaces, and let $S:Y\rightarrow Y$
be a continuous map with an invariant measure $\mu$. Let $\{x_{k}\}\subseteq X$
be a fixed sequence, let $n_{k}\rightarrow\infty$, and suppose that
the sequence $(x_{k},S^{n_{k}}y)_{k=1}^{\infty}$ equidistributes
to a measure $\nu_{y}$ on $X\times Y$ for $\mu$-a.e. $y$. Let
$\nu=\int\nu_{y}d\mu(y)$ and $\tau=\pi_{1}\nu$, where $\pi_{1}(x,y)=x$.
Then $\nu=\tau\times\mu$.
\end{lem}

\begin{proof}
The averages 
\[
\nu_{y,N}=\frac{1}{N}\sum_{k=1}^{N}\delta_{x_{k}}\times\delta_{S^{n_{k}}y}
\]
converges to $\nu_{y}$ for $\mu$-a.e. $y$, therefore $\int\nu_{y,N}d\mu(y)\rightarrow\int\nu_{y}d\mu(y)$.
On the other hand, by $S$-invariance of $\mu$ we have
\begin{align*}
\int\nu_{y,N}d\mu(y) & =(\frac{1}{N}\sum_{k=1}^{N}\delta_{x_{k}})\times\mu
\end{align*}
Thus, the measure$\int\nu_{y,N}d\mu(y)$ is a product measure whose
second marginal is $\mu$, so their limit $\int\nu_{y}d\mu(y)$ has
this form as well.
\end{proof}
Applying this with $X$ the trivial one-point system, we get:
\begin{cor}
If $Y$ is a compact metric space, $S:Y\rightarrow Y$ is continuous,
$\mu$ is an invariant probability measure on $Y$ and $n_{k}\rightarrow\infty$
is such that $(S^{n_{k}}y)_{k=1}^{\infty}$ equidistributes for a
measure $\nu_{y}$ for $\mu$-a.e. $y$, then $\int\nu_{y}d\mu(y)=\mu$.
\end{cor}

\begin{prop}
\label{prop:time-change}Let $(X_{1},\mu_{1},T_{1})$ and $(X_{2},\mu_{2},T_{2})$
be measure preserving systems on standard measure spaces, and fix
real numbers $\beta_{1},\beta_{2}>0$. Then for $\mu_{1}\times\mu_{2}$-a.e.
$(x_{1},x_{2})$, the orbit $(T_{1}^{[\beta_{1}n]}x_{1},T_{2}^{[\beta_{2}n]}x_{2})$
equidistributes for a measure $\nu_{x_{1},x_{2}}$ on $X_{1}\times X_{2}$
satisfying 
\[
\int\nu_{x_{1},x_{2}}d\mu_{1}\times\mu_{2}(x_{1},x_{2})=\mu_{1}\times\mu_{2}
\]
\end{prop}

\begin{rem}
It is important to note that in general $\nu_{x_{1},x_{2}}$ may not
be $T_{1}\times T_{2}$-invariant. Indeed if $(X,\mu,T_{i})$ has
a rational multiple of $1/\beta_{i}$ in its pure point spectrum,
it won't be. 
\end{rem}

\begin{proof}
We prove this first when $X_{2}$ is the trivial (one point system),
so we are dealing with a single transformation $(X,\mu,T)$ and parameter
$\beta>0$. Form the suspension of $(X,\mu,T)$ by the constant function
of height $1$, obtaining the flow $\{\widetilde{T}_{t}\}_{t\in\mathbb{R}}$
on $\widetilde{X}=X\times[0,1]/\sim$ where $\sim$ is the relation
$(x,1)\sim(Tx,0)$, so $\widetilde{T}_{t}$ preserves $\widetilde{\mu}=\mu\times Lebesgue$.
Let $\mathcal{F}\subseteq C(X)$ be a dense countable set and for
$f\in\mathcal{F}$ let $\widetilde{f}\in C(\widetilde{X})$ be given
by $\widetilde{f}(x,t)=f(x)$. Apply the ergodic theorem to the time-$\beta$
map $T_{\beta}=\widetilde{T}_{\beta}$ and the maps $\widetilde{f}$
($f\in\mathcal{F}$) to conclude that the averages 
\[
\frac{1}{N}\sum_{n=1}^{N}\widetilde{f}(T_{\beta}^{n}(x,t))=\frac{1}{N}\sum_{n=1}^{N}f(T^{[\beta n+t]}x)
\]
converge for $\mu$\textendash a.e. $x$ and Lebesgue-a.e. $t\in[0,1]$.
This means that for $\mu$-a.e. $x$ and a.e. $t\in[0,1]$, the point
$(x,t)$ equidistributes for a measure $\widetilde{\nu}_{x,t}$. Next,
note that, endowing $\widetilde{X}$ with the metric induced by the
product metric on $X\times[0,1]$, the orbits of $(x,t)$ and $(x,t')$
are within distance $|t-t'|$ of each other outside a set of density
$|t-t'|$ provided that $|t-t'|<1/2$, and we conclude that $\widetilde{\nu}_{x,0}$
is well-defined for $\mu$-a.e. $x$ (we see this by approximating
the orbit of $(x,0)$ by orbits of $(x,t_{n})$ for Lebesgue-typical
$t_{n}\searrow0$). Thus, $\nu_{x}=\widetilde{\nu}_{x,0}$ is well
defined. The fact that $\int\nu_{x}d\mu(x)=\mu$ follows from the
previous lemma and its corollary.

The generalization to two maps (or more generally, $k$ maps) is proved
in the same way, considering the suspension $\mathbb{R}^{2}$-action
$\{T_{1,s}\times T_{2,t}\}$ on $\widetilde{X}_{1}\times\widetilde{X}_{2}$,
and applying the ergodic theorem to the map $T_{1,\beta_{1}}\times T_{2,\beta_{2}}$.
\end{proof}
\begin{cor}
\label{cor:joining-with-time-change}Let $(X,\mu,T)$ be an ergodic
measure preserving system on a compact metric space. Let $\beta>0$
and $\theta\in\mathbb{R}$. Then for $\mu$-a.e. $x$ the sequence
$(n\theta,T^{[\beta n]}x)$ equidistributes for a measure $\nu_{x}$
on $[0,1)\times X$ that satisfies $\int\nu_{x}d\mu(x)=\tau\times\mu$,
where $\tau$ is the invariant measure on $([0,1),R_{\theta})$ supported
on the orbit closure of $0$.
\end{cor}

\begin{proof}
Apply Proposition \ref{prop:time-change} to $X_{1}=([0,1),R_{\theta},\tau)$
and $X_{2}=(X,\mu,T)$, and with $\beta_{1}=1$ and $\beta_{2}=\beta$.
We conclude that for Lebesgue-a.e. $u\in[0,1)$ and $\mu$-a.e. $x$
(chosen independently), 
\[
(n\theta+u,T^{[n\beta]}x)
\]
equidistributes for a measure $\nu_{u,x}$, and these measures integrate
to $m\times\mu$, where $m$ denotes Lebesgue measure on $[0,1)$.
Since translation in the first coordinate is a continuous action commuting
with the other dynamics, we conclude, by translating the first coordinate
by $-u$ that for $\mu$-a.e. $x$ the sequence
\[
(n\theta,T^{[n\beta]}x)
\]
equidistributes for $\nu_{0,x}=\nu_{x}$. Since translation of the
first coordinate does not affect the projection to the second coordinate,
we conclude also that $\nu_{x}$ projects to $\mu$ on the last coordinate.
Of course, the first coordinate equidistributes for $\tau$. Applying
now Lemma \ref{lem:pairs} we find that $\int\nu_{x}d\mu(x)=\tau\times\mu$.
\end{proof}

\section{\label{sec:The-Fourier-tranform-of-scaled-measures}The Fourier transform
of scaled measures}

We establish some elementary estimates on the Fourier coefficients
of well spread-out measures on the line, when they are scaled by a
random amount.

Let $e(x)=\exp(2\pi ix)$, and for $m\in\mathbb{R}$ we write $e_{m}(x)=e(mx)=\exp(2\pi imx)$.
We write $\widehat{f}$ and $\widehat{\nu}$ for the Fourier transform
of a function or measure, respectively; we also sometimes write it
as
\[
\mathcal{F}_{m}(\nu)=\widehat{\nu}(m)
\]
We define translation and scaling maps of the real line: 
\begin{align*}
R_{\theta}x & =x+\theta\\
S_{t}x & =tx
\end{align*}
We note that $|\widehat{R_{\theta}\nu}|=|\widehat{\nu}|$.
\begin{lem}
\label{lem:smoothing-by-scaling}Let $f\in C^{1}([a,b])$ and suppose
that $\int_{a}^{b}f(x)dx=1$. Then for $\xi\neq0$.
\[
|\widehat{f}(\xi)|<\frac{1}{\pi\xi}(\left\Vert f\right\Vert _{\infty}+(b-a)\left\Vert f'\right\Vert _{\infty})
\]
\end{lem}

\begin{proof}
Using integration by parts,
\begin{align*}
|\widehat{f}(\xi)| & =|\int_{a}^{b}f(x)e_{\xi}(x)dx|\\
 & =|\frac{1}{2\pi i\xi}f(b)e_{\xi}(b)-\frac{1}{2\pi i\xi}f(a)e_{\xi}(a)-\frac{1}{2\pi i\xi}\int_{a}^{b}f'(x)e_{\xi}(x)dx|\\
 & \leq\frac{2}{2\pi\xi}\left\Vert f\right\Vert _{\infty}+\frac{1}{2\pi\xi}\left\Vert f'\right\Vert _{\infty}(b-a)\\
 & \leq\frac{1}{\pi\xi}(\left\Vert f\right\Vert _{\infty}+(b-a)\left\Vert f'\right\Vert _{\infty})
\end{align*}
\end{proof}
\begin{lem}
\label{lem:FT-smoothing-bound}Let $\nu\in\mathcal{P}(\mathbb{R})$
and $b>1$. Then for every $r>0$ and $m\neq0$,
\[
\int_{0}^{1}|\widehat{S_{b^{t}}\nu}(m)|^{2}dt\leq\frac{1}{r\cdot m\cdot\ln b}+\int\nu(B_{r}(y))d\nu(y)
\]
\end{lem}

\begin{proof}
Let $Y,Y'$ be independent random variables with distribution $\nu$,
so $tY$ has distribution $S_{t}\nu$. Using the definition of the
Fourier transform and Fubini,
\begin{align}
\int_{0}^{1}|\widehat{S_{b^{t}}\nu}(m)|^{2}dt & =\int_{0}^{1}|\mathbb{E}(e(mb^{t}Y))|^{2}dt\nonumber \\
 & =\int_{0}^{1}\mathbb{E}(e(mb^{t}Y)\cdot\overline{e(mb^{t}Y')})dt\nonumber \\
 & =\int_{0}^{1}\mathbb{E}(e(mb^{t}(Y-Y')))dt\nonumber \\
 & =\mathbb{E}(\int_{0}^{1}e(mb^{t}(Y-Y'))dt)\label{eq:scaled-FT}
\end{align}
For each pair of values $y,y'$ of $Y,Y'$, the integral $\int_{0}^{1}e(mb^{t}(Y-Y'))dt$
is just the $m(Y-Y')$-th Fourier coefficient of the random variable
$Z$, where $Z$ is the push-forward of the uniform measure on $[0,1]$
by the map $t\mapsto b^{t}$. The density function $f:[1,b]\rightarrow\mathbb{R}$
of $Z$ satisfies $f(z)=(z\ln b)^{-1}$ so $f'(z)=-(z^{2}\cdot\ln b)^{-1}$,
and we get the bounds 
\begin{align*}
\left\Vert f\right\Vert _{\infty},\left\Vert f'\right\Vert _{\infty} & \leq\frac{1}{\ln b}
\end{align*}
which by the previous lemma (using $m(y-y')$ in place of $\xi$)
gives us
\begin{align*}
|\int_{0}^{1}e_{m}(b^{t}(y-y'))dt| & \leq\frac{2}{\pi m\ln(b)(y-y')}\\
 & <\frac{1}{m\ln(b)(y-y')}
\end{align*}
We can now evaluate (\ref{eq:scaled-FT}) as follows: 
\begin{align*}
\int_{0}^{1}|\widehat{S_{b^{t}}\nu}(m)|^{2}dt & \leq\mathbb{E}\left|\int_{0}^{1}e(mb^{t}(Y-Y'))dt\right|\\
 & =\int\int\left|\int_{0}^{1}e(mb^{t}(y-y'))dt\right|\,d\nu(y')\,d\nu(y)\\
 & =\int\left(\int_{\mathbb{R}\setminus B_{r}(y)}\left|\int_{0}^{1}e(mb^{t}(y-y'))dt\right|\,d\nu(y')\right.\\
 & \qquad\quad+\left.\int_{B_{r}(y)}\left|\int_{0}^{1}e(mb^{t}(y-y'))dt\right|\,d\nu(y')\,\right)\,d\nu(y)\\
 & \leq\int\left(\int_{\mathbb{R}\setminus B_{r}(y)}\left(\frac{1}{m\ln(b)(y-y')}\right)\,d\nu(y')+\int_{B_{r}(y)}1\,d\nu(y')\right)d\nu(y)\\
 & \leq\frac{1}{m\ln(b)r}+\int\nu(B_{r}(y))\,d\nu(y)
\end{align*}
as claimed.
\end{proof}

\section{\label{sec:Proof-of-Hosts-theorem-on-01}Proof of Host's theorem
on $[0,1]$ }

We write $T_{a}x=ax\bmod1$. Let $a,b\geq2$ be multiplicatively independent
integers and $\mu$ a $T_{a}$-invariant and ergodic measure on $[0,1]$
with positive entropy.

\subsection{Natural extension, $a$-adic partition, conditional measures}

Let $\mathcal{A}$ denote the partition of $[0,1)$ into intervals
$[k/a,(k+1)/a)$, $k=0,\ldots,a-1$, and $\mathcal{A}_{n}=\bigvee_{i=0}^{n-1}T_{a}^{-i}\mathcal{A}$
the $a$-adic partition of generation-$n$ whose atoms are of the
form $[k/a^{n},(k+1)/a^{n})$ for $k=0,\ldots,a^{n}-1$. We also write
$\mathcal{A}_{n}(x)$ for the atoms of $\mathcal{A}_{n}$ containing
$x$, sometimes identified with the word of the first $n$ $a$-adic
digits in the representation of $x\in[0,1]$. Let $*$ denote concatenation
of sequences.

Let 
\[
\Omega^{-}=\{(\omega_{n})_{n\leq0}\,:\,\omega_{n}\in\{0,\ldots,a-1\}\}
\]
and take the natural extension of $([0,1],\mu,T_{a})$, which we realize
as the space $\widetilde{\Omega}=\Omega^{-}\times[0,1]$ with the
map $\widetilde{T}_{a}(\omega,x)=(\omega*\mathcal{A}(x),T_{a}x)$.
The measure $\mu$ extends uniquely to an ergodic $\widetilde{T}_{a}$-invariant
measure $\widetilde{\mu}$ on $\widetilde{\Omega}$ such that the
projection $(\omega,x)\mapsto x$ is a factor map of $(\widetilde{\Omega},\widetilde{\mu},\widetilde{T}_{a})\rightarrow([0,1],\mu,T_{a})$.
We denote the elements of $\widetilde{\Omega}$ by $\widetilde{\omega}=(\omega,x)$.
We also sometimes identify $\widetilde{\mu}$ with its projection
to $\Omega^{-}$.

Lift the partition $\mathcal{A}$ of $[0,1]$ via the projection $(\omega,x)\mapsto x$
to a partition $\widetilde{\mathcal{A}}$ in $\widetilde{\Omega}$
so that $\widetilde{\mathcal{A}}_{n}=\bigvee_{i=0}^{n-1}\widetilde{T}^{-i}\widetilde{\mathcal{A}}$
is the lift of $\mathcal{A}_{n}$. 

From now on we do not distinguish between $\mathcal{A}$ and $\widetilde{\mathcal{A}}$
and similarly for other partitions.

Let $\mathcal{C}=\bigvee_{i=-\infty}^{0}\widetilde{T}_{a}^{i}\mathcal{\widetilde{A}}$
denote the $\sigma$-algebra in $\widetilde{\Omega}$ generated by
projection ``to the past'', $(\omega,x)\mapsto\omega\in\Omega^{-}$. 

Let $\{\widetilde{\mu}_{\widetilde{\omega}}^{\mathcal{C}}\}_{\widetilde{\omega}\in\Omega}$
denote the corresponding disintegration. We abbreviate $\widetilde{\mu}_{\omega}$
or $\widetilde{\mu}_{\omega,x}$ for $\widetilde{\mu}_{(\omega,x)}^{\mathcal{C}}$,
which does not introduce any ambiguity since since $\widetilde{\mu}_{(\omega,x)}^{\mathcal{C}}$
depends only on the $\Omega^{-}$-component $\omega$ of $(\omega,x)$.
The atoms $\mathcal{C}(\omega,x)=\{\omega\}\times[0,1]$ of $\mathcal{C}$
are naturally identified with $[0,1]$, giving an identification of
$\widetilde{\mu}_{\omega}$ with a measure on $[0,1]$, which we denote
$\mu_{\omega}$. We thus have 
\[
\mu=\int\mu_{\omega}d\widetilde{\mu}(\omega)
\]

By $\mathcal{C}=\bigvee_{i=-\infty}^{0}\widetilde{T}_{a}^{i}\mathcal{A}$,
we have $\mathcal{C}\lor\mathcal{A}_{n}=\widetilde{T}_{a}^{n}\mathcal{C}$,
which gives us the equivariance relation
\begin{equation}
T_{a}^{n}\left((\mu_{\omega})_{\mathcal{A}_{n}(x)}\right)=\mu_{\widetilde{T}_{a}^{n}(\omega,x)}\label{eq:8}
\end{equation}

\subsection{Time change (matching the rates of $T_{a}$ and $T_{b}$)}

Define
\[
\alpha=\frac{\log b}{\log a}
\]
Independence of $a,b$ implies that $\alpha$ is irrational. Set
\[
n'=[\alpha n]
\]
so that $b^{n}\approx a^{n'}$; more precisely, we can write
\[
b^{n}a^{-n'}=a^{\alpha n-[\alpha n]}=a^{z_{n}}
\]
where
\[
z_{n}=\alpha n\bmod1
\]
is the orbit of $0$ under the irrational rotation $R_{\alpha}:z\mapsto z+\alpha\bmod1$
of the compact group $\mathbb{R}/\mathbb{Z}$. 

\subsection{Applying Weyl's criterion and Theorem \ref{thm:equidistribution-1}}

We wish to show that the sequence $\{T_{b}^{n}x\}_{n=1}^{\infty}$
equidistributes for Lebesgue measure for $\mu$ -a.e. $x$. By Weyl's
equidistribution criterion, we must show for $\mu$-a.e. $x$ and
for every $m\in\mathbb{Z}\setminus\{0\}$ that
\[
\lim_{N\rightarrow\infty}\frac{1}{N}\sum_{n=1}^{N}e_{m}(T_{b}^{n}x)=0
\]
where $e_{m}(t)=\exp(2\pi it)$. It is enough to show that for every
$\varepsilon>0$,
\[
\mu\left(x\;:\;\limsup_{N\rightarrow\infty}\left|\frac{1}{N}\sum_{n=1}^{N}e_{m}(T_{b}^{n}x)\right|<\varepsilon\right)>1-\varepsilon
\]
Since $\mu$ is invariant under $T_{a}^{k}$ for all $k$, the last
property follows if we show that for every $\varepsilon>0$ there
is a $k$ such that 
\begin{equation}
\mu\left(x\;:\;\limsup_{N\rightarrow\infty}\left|\frac{1}{N}\sum_{n=1}^{N}e_{m}(T_{b}^{n}\circ T_{a}^{k}x)\right|<\varepsilon\right)>1-\varepsilon\label{eq:11}
\end{equation}

If $\omega$ is chosen according to $\widetilde{\mu}$ and $x$ conditionally
independently according to $\mu_{\omega}$, then $x$ has distribution
$\mu$. Thus, for $\widetilde{\mu}$-typical $\omega$, we want to
apply Theorem \ref{thm:equidistribution-1} to the limit in the last
event for $\mu_{\omega}$-typical $x$ and $\mathcal{A}_{n'}$; we
can do so because $T_{b}^{n}\mathcal{A}_{n'+k}(x)$ has diameter $O(a^{-k})$,
by choice of $n'$. Recalling that $S_{t}$ denotes scaling by $t$,
there exists $c=c(x,n)$ so that $T_{a}^{n'}|_{\mathcal{A}_{n'}(x)}=S_{a^{n'}}+c$.
Writing $\tau_{c}$ for translation by $c$, by (\ref{eq:8}) we have
$T_{a}^{n'}\left((\mu_{\omega})_{\mathcal{A}_{n'}(x)}\right)=\mu_{\widetilde{T}_{a}^{n'}(\omega,x)}$
and hence $(\mu_{\omega})_{\mathcal{A}_{n}(x)}=S_{a^{-n'}}\tau_{-c}\left(\mu_{\widetilde{T}_{a}^{n}(\omega,x)}\right)$.
This gives
\begin{align}
T_{b}^{n}\left(T_{a^{k}}\left((\mu_{\omega})_{\mathcal{A}_{n'}(x)}\right)\right) & =S_{a^{k}}\circ S_{b^{n}}\circ\left(S_{a^{-n'}}\tau_{-c}(\mu_{\widetilde{T}_{a}^{n'}(\omega,x)})\right)\bmod1\nonumber \\
 & =\tau_{-b^{z_{n}}a^{k}c}\left(S_{a^{k}}\circ S_{a^{z_{n}}}(\mu_{\widetilde{T}_{a}^{n'}(\omega,x)}))\right)\bmod1\label{eq:lifting}
\end{align}
Applying Theorem \ref{thm:equidistribution-1} to the limit in (\ref{eq:11}),
inserting the expression above for $T_{b}^{n}\left(T_{a^{k}}(\mu_{\omega})_{\mathcal{A}_{n'}(x)}\right)$
and using the triangle inequality, we have
\begin{align}
\limsup_{N\rightarrow\infty}\left|\frac{1}{N}\sum_{n=1}^{N}e_{m}(T_{b}^{n}T_{a}^{k}x)\right| & =\limsup_{N\rightarrow\infty}\left|\frac{1}{N}\sum_{n=1}^{N}\int e_{m}d(T_{b}^{n}(T_{a}^{k}(\mu_{\omega})_{\mathcal{A}_{n'}(x)}))\right|\nonumber \\
 & =\limsup_{N\rightarrow\infty}\left|\frac{1}{N}\sum_{n=1}^{N}\mathcal{F}_{m}\left(T_{b}^{n}(T_{a^{k}}(\mu_{\omega})_{\mathcal{A}_{n'}(x)})\right)\right|\nonumber \\
 & \leq\limsup_{N\rightarrow\infty}\frac{1}{N}\sum_{n=1}^{N}\left|\mathcal{F}_{m}\left(T_{b}^{n}(T_{a^{k}}(\mu_{\omega})_{\mathcal{A}_{n'}(x)})\right)\right|\nonumber \\
 & =\limsup_{N\rightarrow\infty}\frac{1}{N}\sum_{n=1}^{N}\left|e_{m}(-z_{n}a^{k}c)\cdot\mathcal{F}_{m}\left(S_{a^{k}}S_{a^{z_{n}}}\mu_{\widetilde{T}_{a}^{n'}(\omega,x)}\right)\right|\nonumber \\
 & =\limsup_{N\rightarrow\infty}\frac{1}{N}\sum_{n=1}^{N}\left|\mathcal{F}_{m}\left(S_{a^{z_{n}}}(S_{a^{k}}\mu_{\widetilde{T}_{a}^{n'}(\omega,x)})\right)\right|\label{eq:9}
\end{align}
Thus, it is enough to show that for all $m\in\mathbb{Z}\setminus\{0\}$,
for every $\varepsilon>0$, there exists a $k$ such that 
\begin{equation}
\widetilde{\mu}\left((\omega,x)\;:\;\lim_{N\rightarrow\infty}\frac{1}{N}\sum_{n=1}^{N}\left|\mathcal{F}_{m}\left(S_{a^{z_{n}}}(S_{a^{k}}\mu_{\widetilde{T}_{a}^{n'}(\omega,x)})\right)\right|<\varepsilon\right)>1-\varepsilon\label{eq:10}
\end{equation}

\subsection{Identifying the limit using Proposition \ref{cor:joining-with-time-change}}

The limit in (\ref{eq:10}) can be identified using Corollary \ref{cor:joining-with-time-change}:
with $\tau=$Lebesgue measure, there is a decomposition $\tau\times\widetilde{\mu}=\int\nu_{\omega,x}d\widetilde{\mu}(\omega,x)$
where $\nu_{\omega,x}$ are measures on $(\mathbb{R}/\mathbb{Z})\times\widetilde{\Omega}$
such that 
\[
\lim_{N\rightarrow\infty}\frac{1}{N}\sum_{n=1}^{N}\left|\mathcal{F}_{m}\left(S_{a^{z_{n}}}(S_{a^{k}}\mu_{\widetilde{T}_{a}^{n'}(\omega,x)}\right)\right|=\int\left|\mathcal{F}_{m}\left(S_{a^{z}}S_{a^{k}}\mu_{\eta}\right)\right|d\nu_{\omega,x}(z,\eta)
\]
(we identify $\mathbb{R}/\mathbb{Z}$ with $[0,1)$, which makes the
term $a^{z}$ meaningful. One should note that $\eta\mapsto\mu_{\eta}$
is not continuous as a function on $(\widetilde{\Omega}^{-},\widetilde{\mu})$,
so the equidistribution established in Corollary \ref{cor:joining-with-time-change}
does not formally imply the last equation. But since we are concerned
with $\widetilde{\mu}$-typical $(\omega,x)$, we can just pass to
a model of the system where $\eta\mapsto\mu_{\eta}$ is continuous
and apply the corollary there. For example, add another coordinate
in the shift space $\widetilde{\Omega}$ representing $\mu_{(\omega,x)}$).
Thus, up to changing $\varepsilon$, (\ref{eq:10}) will follow if
we show that for every $\varepsilon>0$, for some $k$
\[
\int\left(\int\left|\mathcal{F}_{m}\left(S_{a^{z}}S_{a^{k}}\mu_{\eta}\right)\right|d\nu_{\omega,x}(z,\eta)\right)d\widetilde{\mu}(\omega,x)<\varepsilon
\]
Using $\int\nu_{\omega,x}d\widetilde{\mu}(\omega,x)=\tau\times\widetilde{\mu}$
and Cauchy-Schwartz, it is enough to show that for every $\varepsilon>0$,
for some $k$
\[
\int\int_{0}^{1}\left|\mathcal{F}_{m}\left(S_{a^{z}}S_{a^{k}}\mu_{\omega}\right)\right|^{2}dz\,d\widetilde{\mu}(\omega)<\varepsilon
\]

\subsection{Estimating the Fourier transform}

The last inequality follows directly from Lemma \ref{lem:FT-smoothing-bound}
applied to $\nu=S_{a^{k}}\mu_{\eta}$ and $r=a^{k/2}$, provided $k$
satisfies $1/(a^{k/2}\cdot m\cdot\ln a)<\varepsilon/2$ and 
\[
\int\nu(B_{r}(y))d\nu(y)<\varepsilon/2
\]
The first inequality holds for all large enough $k$, and the second
inequality, involving $\nu=S_{a^{k}}\mu_{\eta}$, can be re-written
as 
\[
\int\mu_{\eta}(B_{a^{-k/2}}(y))d\mu_{\eta}(y)<\varepsilon/2
\]
This holds for a fixed non-atomic $\mu_{\eta}$ for all large enough
$k$, and therefore for large $k$ it holds with probability arbitrarily
close to one over the choice of $\eta$. Because $\mu$ has positive
entropy, $\mu_{\eta}$ are a.s. non-atomic; this concludes the proof.

\section{\label{sec:Generalizations-of-Host}Generalizations}

Our argument can be generalized in many ways. For example the same
proof applies if $\mu$ is a strongly separated self-similar measure
on $[0,1]$ defined by contractions $f_{i}(x)=r_{i}x+t_{i}$, with
$\log r_{i}/\log b\notin\mathbb{Q}$ for at least one $i$. One chooses
$\mathcal{A}_{n}$ to be the partition of $\mu$ into cylinder measures
of diameter comparable to $b^{-n}$. If all $r_{i}$ are equal, we
find that $T^{n}\mu_{\mathcal{A}_{n}(x)}$ is, up to a translation,
simply $\mu$ scaled by $b^{z_{n}}$, with $(z_{n})$ an irrational
rotation, and the analysis continue as before. If there are distinct
$r_{i}$, then $z_{n}$ is a random sequence coming from the symbolic
coding, and a small additional argument is needed in place of Proposition
\ref{prop:time-change}. The method extends beyond self-similar measures
to positive entropy ergodic measures on the symbolic coding of the
attractor. It is likely that one can show that the methods recovers
the main results of \cite{HochmanShmerkin2015-equidistribution-from-fractal-measures}
concerning so-called uniformly scaling measures generating a scale-invariant
distribution under suitable spectral assumptions. 

In the forthcoming paper \cite{Hochman2019}, this method is extended
to groups of endomorphisms of $\mathbb{T}^{d}$. The precise statements
can be found there.

\bibliographystyle{plain}
\bibliography{bib}

\end{document}

%% file: tex-shortcuts.tex
\newcommand{\hide}[1]{}

\DeclareMathOperator{\diam}{diam}
\DeclareMathOperator{\supp}{supp}

